\documentclass[a4paper,11pt]{article}

\usepackage[sc]{mathpazo}
\usepackage{textcomp}
\usepackage[T1]{fontenc}
\newcommand{\qpltextnumero}{\bgroup\fontfamily{qpl}\selectfont\textnumero\egroup}

\usepackage[
  margin=1in,
  includefoot,
  footskip=30pt,
]{geometry}

\usepackage{amsthm,mathtools,amssymb,amsfonts}
\usepackage[usenames,dvipsnames]{xcolor}

\usepackage{tikz}
\usetikzlibrary{calc,decorations.pathreplacing}
\tikzset{every picture/.style={line width=1pt}}
\tikzset{point/.style={circle,fill,draw,inner sep=0,minimum size=3pt}}
\tikzset{vertex/.style={circle,fill,draw,inner sep=0,minimum size=7pt}}
\tikzset{overtex/.style={circle,fill=none,draw,inner sep=0,minimum size=7pt}}

\usepackage[shortlabels]{enumitem}
\usepackage{subcaption}
\usepackage{booktabs}
\usepackage{hyperref}
\usepackage{bookmark}

\newtheorem{theorem}{Theorem}[section]
\newtheorem{lemma}[theorem]{Lemma}
\newtheorem{cor}[theorem]{Corollary}

\newtheorem{proposition}[theorem]{Proposition}
\theoremstyle{definition}

\theoremstyle{remark}
\newtheorem{rem}{Remark}[section]
\newtheorem*{rem*}{Remark}
\newtheorem{question}{Question}[section]

\newcommand*{\R}{\mathbb{R}}

\newcommand*{\C}{\mathbb{C}}

\newcommand*{\bZ}{\mathbf{Z}}

\newcommand*{\B}{\mathcal{B}}

\newcommand*{\G}{\mathcal{G}}

\newcommand*{\N}{\mathcal{N}}

\newcommand*{\eps}{\varepsilon}
\newcommand{\var}{w}

\usepackage{xstring}
\usepackage{etoolbox}
\usepackage{textcomp}

\newcommand\ZZspacefactor{1.1}
\newcommand\myhpos{c}
\newcommand*{\scriptwidth}[1]{\widthof{$\scriptstyle\mathsurround=0pt#1$}}
\newcommand*{\maxscriptwidth}[2]{\maxof{\scriptwidth{#1}}{\scriptwidth{#2}}*\real{\ZZspacefactor}}
\def\centerinmax#1#2{%
\mathmakebox[\maxscriptwidth{#1}{#2}][\myhpos]{#1}%
}
\def\forlistlooptwo#1#2#3{%
    \ifboolexpr{test{\IfSubStr{#2}{,}} and test{\IfSubStr{#3}{,}}}{%
        \forlistlooptwohelper{#1}#2;#3;%
    }{%
        \ifboolexpr{test{\notblank{#2}} and test{\notblank{#3}}}{%
            #1{#2}{#3}%
        }{}%
    }%
}
\def\forlistlooptwohelper#1#2,#3;#4,#5;{%
    #1{#2}{#4}%
    \forlistlooptwo{#1}{#3}{#5}%
}
\def\ZZ#1#2{%
\bZ^{\forlistlooptwo{\centerinmax}{#1}{#2}}_{\forlistlooptwo{\centerinmax}{#2}{#1}}
}

\title{On zero-free regions for the anti-ferromagnetic Potts model on bounded-degree graphs}
\author{Ferenc Bencs\footnote{HAS Alfr\'ed R\'enyi Institute of Mathematics; Department of Mathematics, Central European University; Korteweg de Vries Institute for Mathematics, University of Amsterdam. Email: \texttt{ferenc.bencs@gmail.com}. Partially supported by the MTA R\'enyi Institute Lend\"ulet Limits of Structures Research Group and by Doctoral Research Support Grant of CEU.} \and Ewan Davies\footnote{Korteweg de Vries Institute for Mathematics, University of Amsterdam. Email: \texttt{maths@ewandavies.org}. Supported by the European Research Council under the European Union's Seventh Framework Programme (FP7/2007-2013) / ERC grant agreement \qpltextnumero{} 339109.} \and Viresh Patel\footnote{Korteweg de Vries Institute for Mathematics, University of Amsterdam. Email: \texttt{vpatel@uva.com}. Supported by the Netherlands Organisation for Scientific Research (NWO) through the Gravitation Programme Networks (024.002.003).} \and Guus Regts\footnote{Korteweg de Vries Institute for Mathematics, University of Amsterdam. Email: \texttt{guusregts@gmail.com}. Supported by a personal NWO Veni grant.}}

\begin{document}
\maketitle

\abstract{%
For a graph $G=(V,E)$, $k\in \mathbb{N}$, and complex numbers $w=(w_e)_{e\in E}$ the partition function of the multivariate Potts model is defined as
\[
\bZ(G;k,w):=\sum_{\phi:V\to [k]}\prod_{\substack{e=uv\in E \\ \phi(u)=\phi(v)}}\var_e,
\]
where $[k]:=\{1,\ldots,k\}$.
In this paper we give zero-free regions for the partition function of the anti-ferromagnetic Potts model on bounded degree graphs. 
In particular we show that for any $\Delta\in \mathbb{N}$ and any $k\geq e \Delta+1$,  there exists an open set $U$ in the complex plane that contains the interval $[0,1)$ such that $\bZ(G;k,w)\neq 0$ for any graph $G=(V,E)$ of maximum degree at most $\Delta$ and any $w\in U^E$. (Here $e$ denotes the base of the natural logarithm.) For small values of $\Delta$ we are able to give better results.

As an application of our results we obtain improved bounds on $k$ for the existence of deterministic approximation algorithms for counting the number of proper $k$-colourings of graphs of small maximum degree.
\\[\baselineskip]
}

\section{Introduction}
The Potts model is an important object in statistical physics generalising the Ising model for magnetism. The partition function of the Potts model captures much information about the model and its study connects several different areas including statistical physics, probability theory, combinatorics and theoretical computer science.

Every graph $G$ (throughout the paper we will always assume graphs are simple) has an associated Potts model partition function defined as follows. Fix $k \in \mathbb{N}$, which will be the number of states or colours. We will consider all functions $\phi: V\to [k]:=\{1,\ldots,k\}$ and often refer to $\phi(v)$ as the colour of $v$. 
For our given graph $G=(V,E)$, we associate a variable $\var_e\in \mathbb{C}$ to each edge $e\in E$.
The $k$-state \emph{partition function of the Potts model} for $G$ is a polynomial in the variables $(\var_e)_{e\in E}$ given by
\[
\bZ(G;k,(\var_e)_{e\in E}):=\sum_{\phi:V\to [k]}\prod_{\substack{uv\in E \\ \phi(u)=\phi(v)}}\var_{uv}.
\]
 If $k$ and the $\var_e$ are clear from the context we simply write $\bZ(G)$. One often considers the `univariate' special case when all $\var_e$ are equal to some $\var\in \mathbb{C}$, in which case we write $\bZ(G;k,w)$ for the partition function.
We note that in statistical physics one parametrises $\var_e=e^{\beta J_e}$ with $\beta$ the inverse temperature and $J_e$ the coupling constant.
The model is called \emph{anti-ferromagnetic} if $w_e\in (0,1)$ (i.e.\ $J_e<0$) for each $e \in E$ and \emph{ferromagnetic} if $w_e>1$ (i.e.\ $J_e>0$) for each $e \in E$.

The study of the location of the complex zeros of the partition function is tradditionally motivated by a seminal result of Lee and Yang~\cite{LY52}, roughly saying that absence of complex zeros near a point on the real axis implies that the model does not undergo a phase transition at this point. 
We should mention here that the zeros we consider are in terms of the variables $(w_e)_{e\in E}$ and  are known as `Fisher zeros'~\cite{Fi65} in the literature, as opposed to `Lee-Yang zeros'~\cite{LY52b}, which are zeros in terms of a different variable representing an external field. 
Another motivation is the algorithmic computation of partition functions which has recently been linked to the location of the complex zeros~\cite{Ba16,PR17}. We discuss this theme in more detail after stating our main result: a new zero-free region for the multivariate anti-ferromagnetic Potts model, which will be proved in Section~\ref{sec:condition}. 

\begin{theorem}\label{thm:interval}
For each $\Delta\in \mathbb{N}$ there exists a constant $c_\Delta \leq e$ and an open set $U\subset\C$ containing the real interval $[0,1)$ such that the following holds. For all graphs $G$ of maximum degree at most $\Delta$, all integers $k\ge k_\Delta: = \lceil c_\Delta \cdot \Delta + 1\rceil$, and for all $(\var_e)_{e\in E}$ such that $w_e \in U$ for each $e\in E$, we have \[\bZ(G; k, (\var_e)_{e\in E})\ne 0.\] 
See Table~\ref{tab:bounds} below for better bounds on $c_\Delta$ and $k_\Delta$ for small values of $\Delta$.
\end{theorem}
\begin{rem}\label{rem:closed interval 1}
We can in fact guarantee an open set $U$ containing the \emph{closed} interval $[0,1]$ under the same conditions as in the theorem above. It is however more convenient to work with $[0,1)$.
In Remark~\ref{rem:closed interval} we indicate how to extend our results to the closed interval.  

We moreover note that while we work with simple graphs in the paper, our result also holds for graphs with multiple edges (loops are not allowed). Our proof of Theorem~\ref{thm:interval} only requires a tiny change to accommodate for this. We leave this for the reader.
\end{rem}

\begin{rem}
From a statistical physics perspective, our results show that the anti-ferromagnetic Potts model does not undergo a phase transition at any temperature on any (infinite) lattice of maximum degree $\Delta$ as long as $k\geq e\Delta+1$. It is interesting to note that phase transitions \emph{do} occur at a finite temperature when $k$ is sublinear in $\Delta$: $k\sim\frac{\Delta}{\log\Delta}$. See~\cite{Huangetal}.
\end{rem}

\begin{rem}
We note that recently (after posting a version of the present paper to the arXiv), Liu, Sinclair and Srivastava~\cite{LSS19} proved an improved version of Theorem~\ref{thm:interval} for the univariate case. They managed to bring down $k_\Delta$ to $2\Delta$.
\end{rem}

\begin{table}[ht]
	\caption{Upper bounds on $c_\Delta$ and the resulting bounds on $k_\Delta$ for small $\Delta$, obtained from Theorem~\ref{thm:ballimprove}.\label{tab:bounds}}
	\begin{center}%
	\begin{tabular}{ccccccccccc}
	\toprule
	$\Delta$ & 3 & 4 & 5 & 6 & 7 & 8 & 9 & 10 & 11 & 12\\
	\midrule
	$c_{\Delta}$ & 1.485 & 1.749 & 1.939 & 2.081 & 2.193 & 2.283 & 2.357 & 2.419 & 2.472 & 2.517 \\
	$k_\Delta$ & 6 & 8 & 11 & 14 & 17 & 20 & 23 & 26 & 29 & 32 \\
	\bottomrule
	\end{tabular}
	\end{center}
\end{table}

\subsection{Related work}
There are several results concerning zero-free regions of the partition function of the Potts model, some of which we discuss below. See e.g.~\cite{Shrocketal,Shrock,Chang1,Chang2,Chang3} for results on the location of the (Fisher) zeros of the partition function of the anti-ferromagnetic Potts model on several lattices, and \cite{So01,FP08,JPS13,Ba16} for results on general (bounded degree) graphs.
Let us say a few words on the latter results and connect these to our present work.

The partition function of the Potts model is a special case of the random cluster model of Fortuin and Kasteleyn~\cite{FK72} which, for a graph $G=(V,E)$ and variables $q$ and $(v_e)_{e\in E}$, is given by
\[
Z(G;q,(v_e)_{e\in E}):=\sum_{F\subseteq E} q^{k(F)} \prod_{e\in F}v_e,
\] 
where $k(F)$ denotes the number of components of the graph $(V,F)$.
Indeed, taking $q=k$ and $v_e=w_e-1$ for each edge $e$, it turns out that $Z(G;q,(v_e)_{e\in E})=\bZ(G;k,(\var_e)_{e\in E})$; see~\cite{So05} for more details and for the connection with the Tutte polynomial.

Almost twenty years ago Sokal~\cite{So01} proved that for any graph $G$ of maximum degree $\Delta\in \mathbb{N}$ there exists a constant $C\leq 7.964$ such that if $|1+v_e|\leq 1$ for each edge $e$, then for any $q\in \mathbb{C}$ such that $|q|\geq C\Delta$ one has $Z(G;q,(v_e)_{e\in E})\neq 0$.
The bound on the constant $C$ was improved to $C\leq  6.907$ by Procacci and Fern\'andez~\cite{FP08}. See also~\cite{JPS13} for results when the condition $|1+v_e|\leq 1$ is removed. In our setting, Sokal's result implies that $\bZ(G; k,(\var_e)_{e\in E})\ne 0$ for  any integer $k > C\Delta$ when every $\var_e$ lies in the unit disk. 

Our main result may be seen as an improvement upon the constant $C$, though in a more restricted setting where, instead of demanding that $\bZ(G; k, (\var_e)_{e\in E})$ is nonzero in the unit disk, we demand that $\bZ(G; k, (\var_e)_{e\in E})$ is nonzero in an open region containing $[0,1)$.
Interestingly, our method of proof is completely different from the approach in \cite{So01,FP08,JPS13}, which is based on cluster expansion techniques from statistical physics.
We prove our results by induction using some basic facts from geometry and convexity, building on an approach developed by Barvinok~\cite{Ba16}.
Previously, Barvinok used this approach in~\cite[Theorem 7.1.4]{Ba16} (improving on~\cite{BaS17}) to show that for each positive integer $\Delta$ there exists a constant $\delta_\Delta>0$ (one may choose e.g.\ $\delta_3=0.18$, $\delta_4=0.13$, and in general $\delta_\Delta=\Omega(1/\Delta)$) such that for any positive integer $k$ and any graph $G$ of maximum degree at most $\Delta$ one has 
\begin{equation}\label{eq:Barvinok}
\bZ(G;k,(w_e)_{e\in E})\neq 0 \text{ provided } |1-w_e|\leq \delta_\Delta \text{ for each edge } e.
\end{equation}
In fact this result is proved in much greater generality, but we have  stated it here just for the Potts model.

While the approach in \cite{Ba16} seems crucially to require that $w_e$ is close to $1$, here we present ideas that allow us to extend the approach in a way that bypasses this requirement. As such the approach may be applicable to other types of models.

\subsection{Algorithmic applications}

Barvinok~\cite{Ba16} recently developed an approach to design efficient approximation algorithms based on absence of complex zeros in certain domains. This gives an additional motivation for studying the location of complex zeros of partition functions.
While it is typically \#P-hard to compute the partition function of the Potts model exactly one may hope to find efficient approximation algorithms (although for certain choices of parameters it is known to be NP-hard to approximate the partition function of the Potts model~\cite{GSV14}).

Combining Theorem~\ref{thm:interval} with Barvinok's approach and results from \cite{PR17}, we obtain the following corollary. We discuss how the corollary is obtained at the end of this section.

\begin{cor}\label{cor:algorithm}
Let $\Delta\in \mathbb{N}$, $w\in [0,1]$ and let $k\ge c_{\Delta} \cdot \Delta+1$.
Then there exists a deterministic algorithm which given an $n$-vertex graph of maximum degree at most $\Delta$ computes a number $\xi$ satisfying  
\[e^{-\eps}\leq \frac{\bZ(G;k,w)}{\xi} \leq e^{\eps}
\]
in time polynomial in $n/\eps$.
\end{cor} 
Corollary~\ref{cor:algorithm} gives us a fully polynomial time approximation scheme (FPTAS) for computing the partition function of the anti-ferromagnetic Potts model (for the right choice of parameters).
In the case when $w=0$, $\bZ(G;k,w)$ is the number of proper $k$-colourings of $G$ and so the corollary gives an FPTAS for computing the number of proper $k$ colourings when 
$k\ge k^{\min}_{\Delta}> c_{\Delta} \cdot \Delta+1$. Lu and Yin~\cite{LY13} gave an FPTAS for this problem when $k \geq 2.58 \Delta+1$; we improve their bound for $\Delta=3,\ldots, 11$. 
We remark that for $\Delta=3$ there is in fact an FPTAS for counting the number of $4$-colourings~\cite{LYZZ17}.
Moreover, there exists an efficient randomised algorithm due to Vigoda~\cite{V00}, which is based on Markov chain Monte Carlo methods, that only requires $k> (11/6) \Delta$ colours. See~\cite{Guillem et.al.} for a very recent small improvement on the constant $11/6$.

\begin{proof}[Proof sketch of Corollary~\ref{cor:algorithm}]
We first sketch Barvinok's algorithmic approach applied to the partition function of the Potts model from which Corollary~\ref{cor:algorithm} is derived.
Suppose we wish to evaluate $\bZ(G;k,w)$ at some point $w\in [0,1)$ for some graph $G$ of maximum degree at most $\Delta$ and positive integer $k\geq c_\Delta \cdot \Delta$.
The first step is to define a univariate polynomial $q(z):=\bZ(G;k,1+z(w-1))$. We then wish to compute $q(1)$.

By Theorem~\ref{thm:interval} combined with \eqref{eq:Barvinok} (cf. Remark~\ref{rem:closed interval 1}) there exists an open region $U'$ that contains $[0,1]$ on which $q$ does not vanish. 
Then we take a disk $D$ of radius slightly larger than $1$ and a fixed polynomial $p$ such that $p(0)=0$ and $p(1)=1$ and such that $D$ is mapped into $U'$ by $p$; see \cite[Section 2.2]{Ba16} for details. 
We next define another polynomial $f$ on $D$ by $f(z)=q(p(z))$. Then $f$ does not vanish on $D$ and hence $\log(f(z))$ is analytic on $D$ and has a convergent Taylor series. 
To approximate $f(1)= \bZ(G;k,w)$ we truncate the Taylor series of $\log(f(z))$ (see \cite[Lemma 2.2.1]{Ba16} for details on where exactly to truncate the Taylor series to get a good approximation), and then we compute these Taylor coefficients.

To compute the Taylor coefficients of $\log(f(z))$ it turns out that it is suffices to compute the low order coefficients of the polynomial $q$, since these can be combined with the coefficients of the polynomial $p$ to obtain the low order coefficients of $f$, from which one can deduce the Taylor coefficients of $\log(f(z))$ via the Newton identities; see \cite[Section 2]{PR17}.
By Theorem 3.2 from~\cite{PR17} the low order coefficients of $q$ can be computed in polynomial time, since, up to an easy to compute multiplicative constant, $q$ is a \emph{bounded induced graph counting polynomial} (\cite[Definition 3.1]{PR17}), as is proved in greater generality in \cite[Section 6]{PR17}. 
\end{proof}

\subsection*{Organisation of the paper}
In the next section we set up some notation and discuss some preliminaries that we need in the proof of our main theorem.
This proof is inspired by Barvinok's proof of \eqref{eq:Barvinok}, and has a similar flavour. 
It is based on induction with a somewhat lengthy and technical induction hypothesis.
For this reason we give a brief sketch of our approach in the next section. 
Section~\ref{sec:ball} then contains an induction for Theorem~\ref{thm:interval}.
This induction contain a condition that is checked in Section~\ref{sec:condition}. 
The proof of Theorem~\ref{thm:interval} follows upon combining the results of Sections~\ref{sec:ball} and~\ref{sec:condition}; see the remark after the statement of Proposition~\ref{pr:constants}. 
In Section~\ref{sec:smallDelta} we slightly modify our induction hypotheses and add another condition to it that allows us to improve our bounds for small values of $\Delta$, as documented in Table~\ref{tab:bounds}. 
We note that this section can be read independently from Sections~\ref{sec:ball} and~\ref{sec:condition}.
We close with some concluding remarks in Section~\ref{sec:conclusion}.

\section{Preliminaries, notation and main idea of the proofs}

In order to prove our results, we will need to work more generally with the partition function of the Potts model with boundary conditions.
For a list $W=w_1\ldots w_m$ of distinct vertices of $V$ and a list $L=\ell_1\ldots\ell_m$ of pre-assigned colours in $[k]$ for the vertices in $W$ the \emph{restricted partition function} $\ZZ{W}{L}(G)$ is defined by
\[
\ZZ{W}{L}(G):=\sum_{\substack{\phi:V\to [k]\\\phi \text{ respects }(W,L)}}\prod_{\substack{uv\in E \\ \phi(u)=\phi(v)}}\var_{uv},
\]
where we say that $\phi$ respects $(W,L)$ if for all $i=1\ldots, m$ we have $\phi(w_i)=\ell_i$.
We call the vertices $w_1,\ldots,w_m$ \emph{fixed} and refer to the remaining vertices in $V$ as \emph{free} vertices. The length of $W$ (resp.\ $L$), written $|W|$ (resp.\ $|L|$) is the length of the list.
Given a list of distinct vertices $W' = w_1\ldots w_m$, and a vertex $u$ (distinct from $w_1, \ldots, w_m$) we write $W = W'u$ for the concatenated list $W = w_1\ldots w_m u$ and we use similar notation $L'\ell$ for concatenation of lists of colours.  We write $\deg(v)$ for the degree of a vertex $v$ and we write $G\setminus uv$ ($G-u$) for the graph obtained from $G$ by removing the edge $uv$ (by removing the vertex $u$). 

In our proofs we often view the restricted partition functions $\bZ^W_L(G)$ as vectors in $\C\simeq\R^2$. The following lemma of Barvinok turns out to be very convenient.
\begin{lemma}[{Barvinok~\cite[Lemma 3.6.3]{Ba16}}]\label{lem:vectorsum}
Let $u_1,\dotsc,u_k\in\R^2$ be non-zero vectors such that the angle between any vectors $u_i$ and $u_j$ is at most $\alpha$ for some $\alpha\in[0,2\pi/3)$. Then the $u_i$ all lie in a cone of angle at most $\alpha$ and 
\[
\Big|\sum_{j=1}^ku_j\Big| \ge \cos(\alpha/2)\sum_{j=1}^k|u_j|\,.
\]
\end{lemma}

Let us now try to explain our approach. It starts with Barvinok's approach from \cite[Section 7.2.3]{Ba16} tailored to the partition function of the Potts model.
Fix a vertex $v$ of the graph $G$. 
Then $\bZ(G)=\sum_{i=1}^k \ZZ{v}{i}(G)$. 
If we can prove that the pairwise angles between $\ZZ{v}{i}(G)$ and $\ZZ{v}{j}(G)$ for all $i,j\in [k]$ are bounded above by $2\pi/3$ then one can conclude by the Lemma~\ref{lem:vectorsum} that $\bZ(G)\neq 0$.
So the idea is to show (using induction on list size) that for any list $W$ of distinct vertices of $G$ and $L$ of pre-assigned colours from $[k]$ where $|W|=|L|$ we have for any vertex $v\notin W$ that the pairwise angles between $\ZZ{W,v}{L,i}(G)$ and $\ZZ{W,v}{L,j}(G)$ are bounded by some $\alpha<2\pi/3$. 

To obtain information about $\ZZ{W,v}{L,i}(G)$, the next step is to fix the neighbours of $v$ and apply a suitably chosen induction hypothesis to all of these neighbours combined with some kind of telescoping argument. 
Suppose for the moment that the degree of $v$ is $1$, and let $u$ be the unique neighbour of $v$.
Then 
\begin{equation}\label{eq:expansionB}
\ZZ{W,v}{L,j}(G)=\sum_{i=1}^k\ZZ{W,v,u}{L,j,i}(G)=\sum_{i\neq j} \ZZ{W,v,u}{L,j,i}(G\setminus uv)+w_{uv}\ZZ{W,v,u}{L,j,j}(G\setminus uv).
\end{equation}
To compare $\ZZ{W,v}{L,j}(G)$ with $\ZZ{W,v}{L,j'}(G)$, Barvinok shows that if $w_{uv}$ is sufficiently close to $1$, then their angle is not too big (if $w_{uv} = 1$ then they are equal) and then the induction can continue.

We however allow $w_{uv}$ to be arbitrarily close to zero, so we need an additional idea: in the induction hypothesis, besides the condition that the angle between two vectors $\ZZ{W,v,u}{L,j,i}(G)$ and $\ZZ{W,v,u}{L,j,i'}(G)$ is small, we add the condition that their lengths should not be too far apart. 
This leads to complications, but fortunately they can be overcome with some additional ideas.
We refer to the next section for the induction statement and the details of the proofs. 
We next collect some tools that we will use.

\quad \\
We will need the following simple geometric facts, which follow from the sine law and cosine law for triangles.
\begin{proposition}\label{prop:simplegeom}
Let $u$ and $u'$ be non-zero vectors in $\R^2$. 
\begin{enumerate}[\textup{(\roman*)}]
\item\label{vectorlength}
If the angle between $u$ and $u'$ is at most $\pi/3$, then $|u-u'|\le \max\{|u|,|u'|\}$.
\item\label{vectorangle}
The angle $\gamma$ between $u$ and $u'$ satisfies $\sin\gamma\le |u-u'|/|u'|$.
\end{enumerate}
\end{proposition}

For $r>0$ and $a\in \mathbb{C}$ we denote by $B(a,r) \subseteq \mathbb{C}$ the open disk of radius $r$ centered at $a$.
For $d\in \mathbb{N}$ we denote 
\[
B(a,r)^d:=\{b_1\dots b_d\mid b_i\in B(a,r) \text{ for } i=1,\ldots, d\} \subseteq \mathbb{C}.
\]

We will need the Grace--Szeg\H{o}--Walsh coincidence theorem, which we state here just for disks. Recall that a polynomial $p$ in variables $x_1,\ldots,x_d$ is called \emph{multi-affine} if for each variable its degree in $p$ at most one. 
\begin{lemma}[Grace--Szeg\H{o}--Walsh]\label{lem:GSW}
Let $p$ be a multi-affine polynomial in the variables $x_1,\ldots,x_d$. Suppose that $p$ is symmetric under permuting the variables. Then for any disk $B\subset \mathbb{C}$, if $\zeta_1,\ldots,\zeta_d\in B$, then there exists $\zeta\in B$ such that
\[
p(\zeta,\ldots,\zeta)=p(\zeta_1,\ldots,\zeta_d).
\]
\end{lemma}
We refer the reader to \cite[Theorem 3.4.1b]{RS02} for a proof of this result, background and related results.
Using the previous result we can show convexity of the set $B(1,r)^d \subseteq \mathbb{C}$ for certain choices of $r$ and $d$.
\begin{lemma}\label{lem:convex image}
Let $d\in \mathbb{N}$.
Then for any $0<r<1/d$ the set $B(1,r)^d$ is convex.
\end{lemma}
\begin{proof}
Define $f:\mathbb{C}\to \mathbb{C}$ by $z\mapsto (1+rz)^d$. Then, by Lemma~\ref{lem:GSW}, $B(1,r)^d$ is the image of $B(0,1)$ under $f$. 
We compute the ratio
\[
\frac{f^{\prime \prime}(z)}{f^\prime(z)}=r(d-1)\left(1+zr\right)^{-1}.
\]
The norm of this ratio is, for any $z\in B(0,1)$, strictly upper bounded by $1$, since $r<1/d$.
This implies that for all $z\in B(0,1)$,
\[
\Re\left(1+z\frac{f^{\prime \prime}(z)}{f^\prime(z)}\right)>0.
\]
A classical result cf.\ \cite[Section 2.5]{D83} now implies that the image of $B(0,1)$ under $f$ is a convex set.
\end{proof}

In Section~\ref{sec:smallDelta} we will also need the following geometric lemma, which we prove below.
\begin{lemma}\label{lem:geom}
 Let $u$ and $u'$ be non-zero vectors in $\R^2$ and $r\ge 1$ real number such that the angle between $u$ and $u'$ is at most $\phi<\pi/3$ and 
 \[r^{-1}\le \frac{|u|}{|u'|}\le r\\.\]
Then 
 \[
  |u-u'|\le \max\left\{2\sin(\phi/2), \sqrt{1+r^{-2}-2r^{-1}\cos\phi}\right\}\cdot \max\left\{|u|,|u'|\right\}.
 \]
\end{lemma}
\begin{proof}
Without loss of generality assume that $|u'|\ge |u|$ and $\arg(u)-\arg(u')=\phi \ge 0$. 
Then we can assume that $u'$ is the point $A$ in Figure~\ref{fig:geom}, the length $OA$ is $|u'|$, the length $OD$ is $r^{-1}|u'|$, and that $u$ lies in the shaded area which we denote by $U$. 

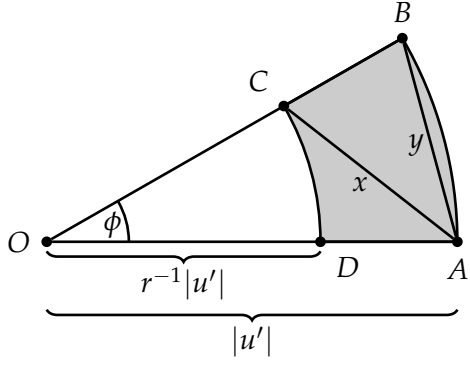
\begin{figure}[ht]
\centering
\begin{tikzpicture}[scale=1.8]
\pgfmathsetmacro{\angle}{30}
	\node[point,label=left:{$O$}] (O) at (0,0) {};
	\node[point,label=below right:{$D$}] (D) at (2,0) {};
	\node[point,label=below:{$A$}] (A) at (3,0) {};
	\node[point,label=above left:{$C$}] (C) at (\angle:2) {};
	\node[point,label=above:{$B$}] (B) at (\angle:3) {};
	
	\draw (O) -- (A);
	\draw (O) -- (B);
	\filldraw[fill opacity=0.2] (0:2cm) -- (0:3cm) arc (0:\angle:3cm) -- (\angle:2cm) arc (\angle:0:2cm) -- cycle;
	\draw (C) -- (A) node[midway,xshift=-4pt,yshift=-4pt] {$x$};
	\draw (B) -- (A) node[midway,xshift=-5pt,yshift=-2pt] {$y$};
	
	\draw (0.6,0) arc(0:\angle:0.6cm);
	\node[label=center:{$\phi$}] at (\angle/2:0.5) {};
	\draw [decorate,decoration={brace,mirror,amplitude=4pt}] ($ (O)+(0,-0.07)  $) -- ($ (D)+(0,-0.07)  $) node [midway,yshift=-12pt] {$r^{-1}|u'|$};
	\draw [decorate,decoration={brace,mirror,amplitude=4pt}] ($ (O)+(0,-0.5)  $) -- ($ (A)+(0,-0.5)  $) node [midway,yshift=-12pt] {$|u'|$};
\end{tikzpicture}
\caption{A diagram for the proof of Lemma~\ref{lem:geom}. The shaded area is $U$.\label{fig:geom}}
\end{figure}

The diameter of $U$ is an upper bound on $|u-u'|$, and it is not hard to see that the diameter of $U$ is the maximum of the distances between any pair of the points $A$, $B$, $C$, and $D$. By symmetry and by the triangle inequality one can see that this maximum is achieved by $x$ or $y$. In order to calculate these lengths we apply the cosine law in the triangles $OAC$ and $OAB$ (and a half-angle formula).
\end{proof}

\section{An induction for Theorem~\ref{thm:interval}}\label{sec:ball}

Let $G=(V,E)$ be a graph together with complex weights $\var=(\var_e)_{e\in E}$ assigned to the edges, a list of distinct vertices $W$, and a list of pre-assigned colours $L$ with $|W| = |L|$ (i.e.\ each vertex in the list $W$ is coloured with the corresponding colour from the list $L$). Recall that the vertices in $W$ are called fixed and those in $V \setminus W$ are called free.  

Let $\eps>0$ be given. We say a neighbour $v$ of a vertex $u \in V$ is a \emph{bad neighbour} of $u$ if $|\var_{uv}| \leq \eps$. We say a colour $\ell\in [k]$ is \emph{good} for a vertex $u \in V$ if  every fixed neighbour of $u$ is not coloured $\ell$; we call $\ell$ \emph{bad} if $u$ has at least one fixed, bad neighbour coloured $\ell$. 
We call a colour \emph{neutral} if it is neither good nor bad. Note that the definition of good, neutral and bad colours also applies if $u$ is fixed. 
We denote the set of good colours by $\G(G,W,L,u)$, the set of neutral colours by $\N(G,W,L,u,\eps)$ and the set of bad colours by $\B(G,W,L,u,\eps)$. We will also write $m(G,W,L,u,\eps, \ell)$ for the number of fixed bad neighbours of $u$ with colour $\ell$.
 When $G$, $W$, $L$, $u$, and $\eps$ are clear from the context we will write e.g.\ $\G=\G(G,W,L,u)$, $\B=\B(G,W,L,u,\eps)$, $\N=[k]\setminus(\G\cup \B)$, and $m(\ell)=m(G,W,L,u,\eps,\ell)$.

For a graph $G=(V,E)$ we call $W \subseteq V$ a \emph{leaf-independent set} if $W$ is an independent set and every vertex in $W$ has degree exactly $1$. In particular this means every vertex in $W$ has exactly one neighbour in $V \setminus W$.

\begin{theorem}
\label{thm:ball2}
Let $\Delta\in \mathbb{N}_{\geq 3}$.
Suppose that $k>\Delta$ and $0<\eps<1$ are such that there exists a positive constant $K< 1/(\Delta-1)$ with $\theta:=\arcsin(K) \in (0,\frac{\pi}{3(\Delta-1+\eps)})$ such that for each $d=0,\ldots,\Delta-1$, with $b=\Delta-d$,
\begin{align}\label{eq:the conditions2}
&0<\frac{(1+\varepsilon)^2}{(k-b)(1-K)^d-\varepsilon b}\leq K.
\end{align}
Then for each graph $G=(V,E)$ of maximum degree at most $
\Delta$ and every $\var=(\var_e)_{e\in E}$ satisfying for each $e \in E$ that
\begin{itemize}
\item[(i)] $|\var_e|\leq \eps$, or 
\item[(ii)] $|\arg(\var_e)|\leq \eps \theta$ and $\eps < |\var_e|\leq 1$, 
\end{itemize} 
 the following statements hold for $\bZ(G) = \bZ(G;k, \var)$.
\begin{enumerate}[\bf{\Alph*}]
\item\label{A}
For all lists $W$ of distinct vertices of $G$ such that $W$ forms a leaf-independent set in $G$ 
and for all lists of pre-assigned colours $L$ of length $|W|$, $\ZZ{W}{L}(G)\neq 0$.
\item\label{B}
For all lists $W=W'u$ of distinct vertices of $G$ such that $W$ is a leaf-independent set 
and for any two lists $L'\ell$ and $L'\ell'$ of length $|W|$:
\begin{enumerate}[\textbf{\textup{(\roman*)}}]
\item\label{Bi}
the angle between the vectors $\ZZ{W',u}{L',\ell}(G)$ and $\ZZ{W',u}{L',\ell'}(G)$ is at most $\theta$,
\item\label{Bii}
\begin{equation}
\frac{\ZZ{W',u}{L',\ell}(G)}{\ZZ{W',u}{L',\ell'}(G)}\in B(1,K).
\end{equation}
\end{enumerate}
\item\label{C}
For all lists $W=W'u$ of distinct vertices such that the initial segment $W'$ forms a leaf-independent set in $G$ 
and for all lists of pre-assigned colours $L'$ of length $|W'|$, the following holds. 
Write $\G=\G(G,W',L',u)$ and $\N=\N(G,W',L',u,\eps)$, let 
 $d$ be the number of free neighbours of $u$, and let $b=\Delta-d$. 
Then 
\begin{enumerate}[\textbf{\textup{(\roman*)}}]
\item\label{Ci}
for any $\ell\in \mathcal{G}\cup \N$, $\ZZ{W',u}{L',\ell}(G)\neq 0$, 
\item\label{Cii}
for any $\ell,\ell'\in \mathcal{G}\cup \N$, the angle between the vectors $\ZZ{W',u}{L',\ell}(G)$ and $\ZZ{W',u}{L',\ell'}(G)$ is at most $(d+b\eps)\theta$,
\item\label{Ciii}
for any $\ell,j\in \G$,
\begin{equation}
\frac{\ZZ{W',u}{L',\ell}(G)}{\ZZ{W',u}{L',j}(G)} \in B(1,K)^d.
\end{equation}
\end{enumerate}
\end{enumerate}
\end{theorem}

\subsection{Proof}
We prove that \ref{A}, \ref{B}, and \ref{C} hold by induction on the number of free vertices of a graph.
The base case consists of graphs with no free vertices. Clearly \ref{A} and \ref{B} hold in this case as they are both vacuous: if there are no free vertices then $W=V$ but then $W$ cannot be a leaf-independent set.

For statement \ref{C} we note that since there are no free vertices, $V\setminus W' = \{u\}$, and hence $G$ must be a star with centre $u$. 
Part \ref{C}\ref{Ci} follows since when $\ell\in \mathcal{G}\cup \N$ we have that  $\ZZ{W',u}{L',\ell}(G)$ is a product over nonzero edge-values.
Part \ref{Cii} follows since changing the colour of $u$ from $\ell$ to $j\in \G\cup\N$, we can obtain $\ZZ{W',u}{L',j}(G)$ from $\ZZ{W',u}{L',\ell}(G)$ by multiplying and dividing by at most $\deg(u)$ factors $w_{uv}$ with $\arg(w_{uv}) \leq\eps \theta$; hence the restricted partition function changes in angle by at most $\Delta \eps\theta$.
Part \ref{Ciii} follows similarly, as when there are no free vertices we must have $d=0$, and changing the colour of $u$ from $j$ to $\ell$ does not change the value of the restricted partition function since both colours are good. Hence $\ZZ{W',u}{L',\ell}(G)/\ZZ{W',u}{L',j}(G)=1\in B(1,K)^d$.

 Now let us assume that statements \ref{A}, \ref{B}, and \ref{C} hold for all graphs with $f\geq 0$ free vertices. We wish to prove the statements for graphs with $f+1$ free vertices. 
We start by proving \ref{A}.

\subsubsection{Proof of \ref{A}}
Let $u$ be any free vertex.
We proceed using the fact that $|\ZZ{W}{L}(G)|=|\sum_{j=1}^k\ZZ{W,u}{L,j}(G)|$. 
Let $\G=\G(G,W,L,u)$, $\B=\B(G,W,L,u,\eps)$, $\N=[k]\setminus(\G\cup\B)$ and $\hat b=|\B|$.
Let $d$ be the number of free neighbours of $u$ and let $b=\Delta-d$. Note that $\hat b\leq b$ and $|\G|\geq k-b$.
After fixing $u$ to any $j\in[k]$ we have one less free vertex, and hence can apply \ref{C} using induction as necessary. 

There are two cases to consider. If $\hat b=0$ then 
by induction using \ref{C}\ref{Ci} we have that the $\ZZ{W,u}{L,j}(G)$ are non-zero and by \ref{C}\ref{Cii}  the angle between any two of the $\ZZ{W,u}{L,j}(G)$ is at most $\Delta\theta<\frac{\Delta}{\Delta-1}\pi/3\le 3/2\cdot \pi/3=\pi/2$. So the $\ZZ{W,u}{L,j}(G)$ all lie in some cone of angle at most $\pi/2$. In particular their sum must be in that cone and nonzero.

If $\hat b>0$ then $u$ must have at least one fixed neighbour, and hence $d\le\Delta-1$. 
Let $H$ be the graph obtained from $G$ by deleting all fixed neighbours of $u$, i.e.\ $H=G-(N_G(u)\cap W)$, and let $W'=W\setminus N_G(u)$ and $L'$ be the sublist of $L$ corresponding to the vertices in $W'$.
Observe that by definition for any $j\in[k]$ we have  
\begin{equation}\label{eq:from G to H in A}
\ZZ{W,u}{L,j}(G) =\ZZ{W',u}{L',j}(H)\cdot \prod_{\substack{v'\in W\cap N_G(u)\\\text{s.t. }L(v')=j }}\var_{uv'},
\end{equation}
where by $L(v')$ we mean the colour that the list $L$ pre-assigns to the vertex $v'$.
In particular, if $j\in \G$, then $\ZZ{W,u}{L,j}(G) =\ZZ{W',u}{L',j}(H)$.
Note also that by construction $u$ has no fixed neighbours in the graph $H$ and hence any colour is good for $u$ in $H$.
Let 
\begin{align*}
M:=\max\big\{\big|\ZZ{W',u}{L',j}(H)\big|: j \in [k]\},
\end{align*}
and assume that $j_M\in[k]$ achieves the maximum above. Note that $M>0$ by induction using \ref{C}\ref{Ci}. 
We then have by the triangle inequality
\begin{align*}
|\ZZ{W}{L}(G)/M|=\Big|\sum_{j=1}^k \ZZ{W,u}{L,j}(G)/\ZZ{W',u}{L',j_M}(H)\Big|
  \ge \Big|\sum_{j\in\G\cup\N} &\ZZ{W,u}{L,j}(G)/\ZZ{W',u}{L',j_M}(H)\Big| 
  \\&- \sum_{j\in\B} |\ZZ{W,u}{L,j}(G)/\ZZ{W',u}{L',j_M}(H)|.
\end{align*}
Since by induction using \ref{C}\ref{Cii} the pairwise angles between the $ \ZZ{W,u}{L,j}(G)$ for $j\in \G\cup \N$ are bounded by $(d+b\eps)\theta\leq (\Delta - 1 + \eps)\theta \leq \pi/3$ these vectors lie in a cone of angle at most $\pi/3$ and therefore, 
\[
\Big|\sum_{j\in\G\cup\N} \ZZ{W,u}{L,j}(G)/\ZZ{W',u}{L',j_M}(H)\Big|\geq \Big|\sum_{j\in\G} \ZZ{W,u}{L,j}(G)/\ZZ{W',u}{L',j_M}(H)\Big|.
\]
By induction using \ref{C}\ref{Ciii}, the numbers $\ZZ{W,u}{L,j}(G)/\ZZ{W',u}{L',j_M}(H)=\ZZ{W',u}{L',j}(H)/\ZZ{W',u}{L',j_M}(H)$ are contained in $B(1,K)^d$, for $j\in \G$.
By Lemma~\ref{lem:convex image} this is a convex set, as $K<1/d$.
Therefore, $\sum_{j\in\G} \ZZ{W,u}{L,j}(G)/\ZZ{W',u}{L',j_M}(H)\in |\G|\cdot B(1,K)^d$, which implies by convexity of $B(1,K)^d$ that
\[
\Big|\sum_{j\in\G} \ZZ{W,u}{L,j}(G)/\ZZ{W',u}{L',j_M}(H)\Big|\geq |\G|\cdot (1-K)^{d}.
\]
By \eqref{eq:from G to H in A} and the definition of $\mathcal{B}$, we have that for each $j\in \B$
\[
|\ZZ{W,u}{L,j}(G)|/|\ZZ{W',u}{L',j_M}(H)|\leq \eps.
\]
Combining these inequalities we arrive at
\[
|\ZZ{W}{L}(G)/M|\geq (k-b)(1-K)^{d}-\eps \hat{b} \geq (k-b)(1-K)^{d}-\eps b.
\]
Now the conditions~\eqref{eq:the conditions2} give that $\ZZ{W}{L}(G)\ne0$.

Next we will prove \ref{B}.

\subsubsection{Proof of \ref{B}}
Since $W = W'u$ is a leaf-independent set, $\deg(u)=1$ and the unique neighbour of $u$, which we call $v$, is free.
We start by introducing some notation.

We define complex numbers $z_j$ for $j\in [k]$ by
\begin{equation}
z_j:=\ZZ{W',u,v}{L',\ell,j}(G\setminus uv)=\ZZ{W',u,v}{L',\ell',j}(G\setminus uv)\,,
\end{equation}
where the second equality holds because $u$ is isolated in $G\setminus uv$. 
Let $w:=w_{uv}$ and define complex numbers $x_j$ and $y_j$ for $j\in [k]$ by
\begin{align*}
x_j&=\begin{cases}
z_j & \text{ if } j\neq \ell;\\
\var z_\ell &\text{ if } j=\ell\,,
\end{cases}&
y_j=\begin{cases}
z_j & \text{ if } j\neq \ell';\\
\var z_{\ell'} &\text{ if } j=\ell'\,.
\end{cases}
\end{align*}
Let $x=\sum_{j=1}^k x_j$ and $y=\sum_{j=1}^k y_j$.
Observe that $x=\ZZ{W',u}{L',\ell}(G)$ and $y=\ZZ{W',u}{L',\ell'}(G)$, and that we may apply induction to the restricted partition function evaluations represented by the $z_j$ because there are $f$ free vertices in $G\setminus uv$ when the vertices in $W'uv$ are fixed. 

For \ref{B}\ref{Bi} and \ref{Bii} we wish to bound the angle between $x$ and $y$ and to constrain the ratio $x/y$ respectively.
To do this we first bound $|y|$ and $|x-y|$. 

We note for later that by the definition of $x$ and $y$, we have $x-y=(\var-1)z_\ell+(1-\var)z_{\ell'}=(1-\var)(z_{\ell'}-z_\ell)$. Also $|1 - \var| \leq 1+ \eps$ by conditions (i) and (ii) in the statement of the theorem so $|x-y| \leq (1+\eps)|z_{\ell'}-z_\ell|$.

Let $\G=\G(G,W'u,L'\ell',v)$, $\B=\B(G,W'u,L'\ell',v,\eps)$, $\N=[k]\setminus(\G\cup\B)$, $\hat b=|\B|$, and suppose $v$ has $d$ free neighbours (in $G$ when $W'u$ is fixed). 
Let $H$ be the graph obtained from $G$ by deleting all fixed neighbours of $v$, i.e.\ $H=G-(N_G(v)\cap W)$, and let $W''=W\setminus N_G(v)$ and $L''$ be the sublist of $L$ corresponding to the vertices in $W''$.
Observe that by definition for any $j\in[k]$ we have  
\begin{equation}\label{eq:from G to H}
 z_j =\ZZ{W',v}{L',j}(G-u) =\ZZ{W'',v}{L'',j}(H)\cdot \prod_{\substack{v'\in W'\cap N_G(v)\\\text{s.t. }L'(v')=j }}\var_{vv'},
\end{equation}
where by $L'(v')$ we mean the colour that the list $L'$ pre-assigns to the vertex $v'$.
In particular, if $j\in \G$, then $\ZZ{W',v}{L',j}(G-u) =\ZZ{W'',v}{L'',j}(H)$.
Note also that by construction $v$ has no fixed neighbours in the graph $H$.
Now write $b=\Delta-d$. Note that $d\le\Delta-1$, and define $M$, $j^*$ by
\begin{align*}
M:=\max\big\{\big|\ZZ{W'',v}{L'',j}(H)\big|: j\in[k] \big\}=|\ZZ{W'',v}{L'',j^*}(H)|.
\end{align*}
We perform a similar calculation to the case $\hat b>0$ of \ref{A} to show that
\begin{equation}\label{eq:bound on y}
|y/M| \geq (k-b)(1-K)^{d}-\hat{b}\eps.
\end{equation}

To see this we have by the triangle inequality,
\begin{align*}
\left|{y}/{\ZZ{W'',v}{L'',j^*}(H)}\right| = \left|\sum_{j=1}^k{\ZZ{W',u,v}{L',\ell',j}(G)}/{\ZZ{W'',v}{L'',j^*}(H)}\right| 
 \ge\ & \left|\sum_{j\in\G\cup\N}\ZZ{W',u,v}{L',\ell',j}(G)/{\ZZ{W'',v}{L'',j^*}(H)}\right|
 \\
 \quad -&\sum_{j\in\B}\left|\ZZ{W',u,v}{L',\ell',j}(G)/{\ZZ{W'',v}{L'',j^*}(H)}\right|\,.
\end{align*}
As before, by \eqref{eq:from G to H} and by induction using \ref{C}\ref{Cii} the pairwise angles of the summands in the sum over $\G\cup\N$ is at most $(d+b\eps)\theta\leq \pi/3$. 
This implies that these numbers lie in a cone of angle at most $\pi/3$, which implies that
\[
\left|\sum_{j\in\G\cup\N}{\ZZ{W',u,v}{L',\ell',j}(G)}/{\ZZ{W'',v}{L'',j^*}(H)}\right|\geq \left|\sum_{j\in\G}{\ZZ{W',u,v}{L',\ell',j}(G)}/{\ZZ{W'',v}{L'',j^*}(H)}\right|.
\]
Now for any $j\in \G$, we have that ${\ZZ{W',u,v}{L',\ell',j}(G)}/{\ZZ{W'',v}{L'',j^*}(H)} = {\ZZ{W'',v}{L'',j}(H)}/{\ZZ{W'',v}{L'',j^*}(H)} \in B(1,K)^d$ by induction using \ref{C}\ref{Ciii}. 
As this set is convex, we have
\[
\left|\sum_{j\in\G}{\ZZ{W',u,v}{L',\ell',j}(G)}/{\ZZ{W'',v}{L'',j^*}(H)}\right|\geq  |\G|(1-K)^d.
\]
Since for any $j\in \B$ we have $m(j)\geq 1$, it follows by \eqref{eq:from G to H} and the definition of the $y_j$ that
\[
\sum_{j\in\B}\left|\ZZ{W',u,v}{L',\ell',j}(G)/{\ZZ{W'',v}{L'',j^*}(H)}\right|
\leq 
\hat{b} \eps^{m(j)} \leq \hat{b}\eps.
\]
Combining these two bounds we obtain \eqref{eq:bound on y}.

We next claim that
\begin{equation}\label{eq:bound x-y over y1}
\frac{|x-y|}{|y|}< K.
\end{equation}

To prove this we need to distinguish two cases, depending on whether or not $\ell$ or $\ell'$ is a bad colour in $G-u$ for the vertex $v$. We first introduce further notation. Let $\widehat\G=\G(G-u,W',L',v)$, $\widehat\B=\B(G-u,W',L',v,\eps)$, $\widehat\N= \N(G-u,W',L',v,\eps)$, and let $\widehat m(j)$ be the number of bad neighbours of $v$ in $G-u$ with pre-assigned colour $j$. Note that $v$ has $d \leq \Delta - 1$ free neighbours in $G-u$.
We now come to the two cases: either both $\ell ,\ell' \in\widehat\G\cup \widehat\N$, or at least one is in $\widehat\B$. 
In the first case, by induction using \ref{C}\ref{Cii} for $\ZZ{W',v}{L',j}(G-u) = \ZZ{W',v}{L',j}(G \setminus uv)=z_j $, the angle between $z_{\ell}$ and $z_{\ell'}$ is at most $(d+b\eps)\theta \le (\Delta - 1 + \eps)\theta \le\pi/3$, and hence we have
$|z_{\ell'}-z_\ell|\leq \max\{|z_\ell|,|z_{\ell'}|\}$ by Proposition~\ref{prop:simplegeom}.
Putting the established bounds together, we have
\begin{equation}
\label{eq:modify1}
\frac{|x-y|}{|y|} 
\le \frac{(1 + \eps) |z_{\ell'} - z_{\ell}| }{|y|}
\leq (1 + \eps)\max_{j\in\{\ell,\ell'\}}\frac{|z_j|/M}{|y|/M}\le \frac{1+\eps}{(k-b)(1-K)^d-\eps b}< K\,,
\end{equation}
where the second inequality follows using \eqref{eq:bound on y} and the definition of $M$, and the final inequality follows from the condition \eqref{eq:the conditions2}.
Hence \eqref{eq:bound x-y over y1} holds when $\ell,\ell'\in\widehat\G \cup \widehat\N$.
 
For the other case, when at least one of $\ell,\ell'$ is in $\widehat\B$, we use the triangle inequality and \eqref{eq:from G to H} to obtain
\begin{align*}
\frac{|z_\ell-z_{\ell'}|}{M}\le \frac{|z_\ell|}{M}+\frac{|z_{\ell'}|}{M}
\le \big(\varepsilon^{\widehat m(\ell)} +\varepsilon^{\widehat m(\ell')} \big)
\le (1+\varepsilon),
\end{align*}
since at least one of $\widehat m(\ell)$ and $\widehat m(\ell')$ is at least $1$ in this case. Therefore, using \eqref{eq:bound on y},
\begin{equation}
\label{eq:modify2}
 \frac{|x-y|}{|y|}
 \le \frac{(1 + \eps) |z_{\ell'} - z_{\ell}|/M }{|y|/M}
  \leq   \frac{(1+\varepsilon)^2}{(k-b)(1-K)^d-\eps b}< K,
\end{equation}
where the final inequality comes from the condition \eqref{eq:the conditions2}, establishing \eqref{eq:bound x-y over y1}.

Now, by Proposition~\ref{prop:simplegeom}, the angle $\gamma$ between $x$ and $y$ satisfies $\sin \gamma\le |x-y|/|y|< K$, and we conclude that $\gamma\le\arcsin(K)=\theta$ as required for \ref{B}\ref{Bi}. 
Additionally, we have 
\[
\frac{x}{y}= \frac{y+x-y}{y}=1+\frac{x-y}{y}\in B(1,K)\,,
\]
since $|x-y|/|y|<K$.
This gives \ref{B}\ref{Bii}. We now turn to \ref{C}.

\subsubsection{Proof of \ref{C}}
We start with \ref{Ci}, that is we will show that for any $\ell\in\G$, $\ZZ{W',u}{L',\ell}(G)\neq 0$. Since we have already proved \ref{A} and \ref{B} for the case of $f+1$ free vertices and since we have $f+1$ free vertices for $\ZZ{W',u}{L',\ell}(G)$, we might hope to immediately apply \ref{A}; the only problem is that $W'u$ is not a leaf- independent set, so we will modify $G$ first. 

Let $v_1,\dots,v_d$ be the free neighbours of $u$. We construct a new graph $H$ from $G$ by adding vertices $u_1, \ldots, u_d$ to $G$ and replacing each edge $uv_i$ with $u_iv_i$ for $i=1, \ldots, d$, while keeping all other edges of $G$ unchanged (so note that $u$ is only adjacent to its fixed neighbours in $H$). Each edge $e$ of $H$ is assigned value $\var'_e$ where $\var'_e = \var_e$ if $e$ is an edge of $G$ and $\var'_{u_iv_i} = \var_{uv_i}$ for the new edges $uv_i$. 
See Figure~\ref{fig:ballH} for an illustrative example.

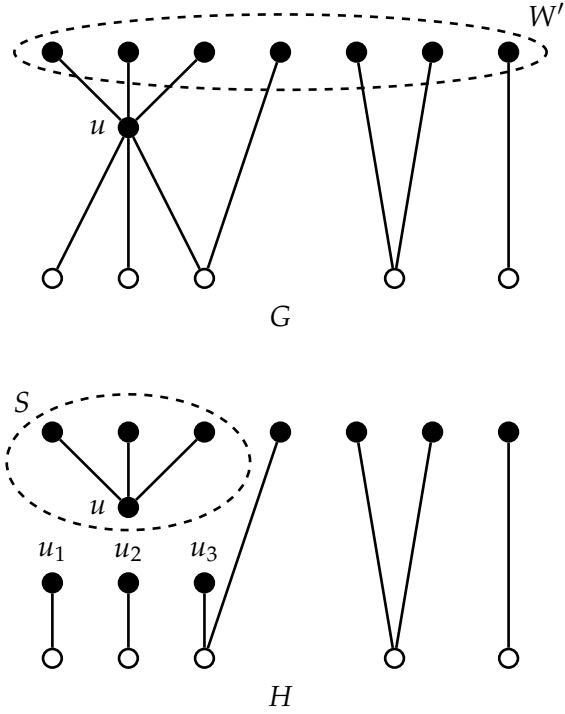
\begin{figure}[ht]
\centering
\begin{tikzpicture}
	\useasboundingbox (-1,1) rectangle (7,-4);
	\foreach \i in {1,2,...,7} {
	\node[vertex] (w\i) at (\i-1, 0) {};
	};
	
	\node[vertex,label=left:{$u$}] (u) at ($ (w2) + (0,-1) $) {};
	
	\foreach \i in {1,2,3} {
	\node[overtex] (f\i) at (\i-1, -3) {};
	};
	\node[overtex] (f4) at (4.5, -3) {};
	\node[overtex] (f5) at (6, -3) {};
	
	\draw (u)--(w1);
	\draw (u)--(w2);
	\draw (u)--(w3);
	\draw (u)--(f1);
	\draw (u)--(f2);
	\draw (u)--(f3);
	\draw (f3)--(w4);
	\draw (f4)--(w5);
	\draw (f4)--(w6);
	\draw (f5)--(w7);
	
	\draw[dashed] (3,0) ellipse (3.5cm and 0.5cm);
	\node (w) at (6.5,0.5) {$W'$};
	\node (G) at (3,-3.5) {$G$};
\end{tikzpicture}
\\
\begin{tikzpicture}
	\useasboundingbox (-1,1) rectangle (7,-4);
	\foreach \i in {1,2,...,7} {
	\node[vertex] (w\i) at (\i-1, 0) {};
	};
	
	\node[vertex,label=left:{$u$}] (u) at ($ (w2) + (0,-1) $) {};
	
	\foreach \i in {1,2,3} {
	\node[overtex] (f\i) at (\i-1, -3) {};
	\node[vertex,label=above:{$u_\i$}] (u\i) at ($ (f\i) + (0,1) $) {};
	};
	\node[overtex] (f4) at (4.5, -3) {};
	\node[overtex] (f5) at (6, -3) {};
	
	\draw (u)--(w1);
	\draw (u)--(w2);
	\draw (u)--(w3);
	\draw (u1)--(f1);
	\draw (u2)--(f2);
	\draw (u3)--(f3);
	\draw (f3)--(w4);
	\draw (f4)--(w5);
	\draw (f4)--(w6);
	\draw (f5)--(w7);
	
	\draw[dashed] ($ (u)!0.6!(w2) $) ellipse (1.6cm and 0.9cm);
	\node (s) at ($ (w1) + (-0.4,0.4) $) {$S$};
	\node (H) at (3,-3.5) {$H$};
\end{tikzpicture}
\caption{An illustration of the construction of $H$ (below) from $G$ (above) in the proof of~\ref{C}. Note that $W'$ forms a leaf-independent set, but that we do not require that $W'u$ has this property.}
\label{fig:ballH}
\end{figure}

Then by construction we have 
\begin{equation}\label{eq:G to H1}
\ZZ{W',u}{L',\ell}(G)=\ZZ{W',u,u_1,\ldots,u_d}{L',\ell,\ell,\ldots,\ell}(H).
\end{equation}

Notice that in $H$, the vertex $u$ together with its neighbours form a star $S$ that is disconnected from the rest of $H$ (and all vertices of $S$ are in $W = W'u$ so they are fixed). Thus $H$ is the disjoint union of $S$ and some graph $\widehat{H}$. 
Thus the partition function $z:=\ZZ{W',u,u_1,\ldots,u_d}{L',\ell,\ell,\ldots,\ell}(H)$ factors as 
\begin{equation}
\label{eq:star1}
\ZZ{W',u,u_1,\ldots,u_d}{L',\ell,\ell,\ldots,\ell}(H)
= \ZZ{W',u_1,\ldots,u_d}{L',\ell,\ldots,\ell}(\widehat{H}) \cdot
\ZZ{W',u}{L',\ell}(S);
\end{equation}
here we abuse notation by having a list $W'u_1\ldots u_d$ (resp. $W'u$) that may contain vertices not in $\widehat{H}$ (resp. $S$); such vertices and their corresponding colour should simply be ignored.

The fixed vertices in $\widehat{H}$ form a leaf-independent set, so we can apply $\ref{A}$ to conclude that the first factor above is nonzero. It is also clear that second factor above is nonzero because all vertices in $S$ are fixed and $\ell\in \mathcal{G}\cup \N$. Hence $z \not=0$ as required.

To prove part \ref{Cii}, we will apply \ref{B} to $\widehat{H}$ with $W'u_1\ldots u_d$ fixed, which (as above) is possible since we already proved \ref{B} for $f+1$ free vertices and $W'u_1\ldots u_d$ restricted to $\widehat{H}$ is a leaf-independent set. 
By \ref{B}\ref{Bi} the angle between
\[
\ZZ{W',u_1,\ldots,u_{d-1},u_d}{L',\ell,\ldots,\ell,\ell'}(\widehat{H})\qquad\text{and}\qquad\ZZ{W',u_1,\ldots,u_{d-1},u_d}{L',\ell,\ldots,\ell,\ell}(\widehat{H})
\]
is at most $\theta$. Continuing to change the label of each $u_i$ one step at the time, we conclude that the angle between
\[
\ZZ{W',u_1,\ldots,u_d}{L',\ell',\ldots,\ell'}(\widehat{H})\qquad\text{and}\qquad\ZZ{W',u_1,\ldots,u_d}{L',\ell,\ldots,\ell}(\widehat{H})
\]
is at most $d\theta$.
We next notice that since for \ref{Cii} we assume $\ell,\ell' \in \G \cup \N$, changing the colour of $u$ from $\ell$ to $\ell'$ can only change $\ZZ{W',u}{L',\ell}(S)$ by $\deg_S(u) \leq \Delta-d=b$ factors, each of argument at most $\eps \theta$ thus giving a total change of angle by at most $b \eps \theta$.
Hence by \eqref{eq:star1}, we therefore conclude that the angle between $\ZZ{W',u}{L',\ell}(G)$ and $\ZZ{W',u}{L',\ell'}(G)$ is at most $d\theta+b\eps\theta$.

To prove \ref{Ciii}
we observe that we can write for any $j,\ell\in [k]$ the telescoping product:
\begin{equation}\label{eq:telescope prime}
\frac{\ZZ{W',u_1,\ldots,u_d}{L',\ell,\ldots,\ell}(\widehat{H})}{\ZZ{W',u_1,\ldots,u_d}{L',j,\ldots,j}(\widehat{H})}=
\frac{\ZZ{W',u_1,\ldots,u_d}{L',\ell,\ldots,\ell}(\widehat{H})}{\ZZ{W',u_1,\ldots, u_{d-1},u_d}{L',\ell,\ldots,\ell,j}(\widehat{H})}\cdots
\frac{\ZZ{W',u_1,u_2,\ldots,u_d}{L',\ell,j,\ldots,j}(\widehat{H})}{\ZZ{W',u_1,\ldots,u_d}{L',j,\ldots,j}(\widehat{H})}\,.
\end{equation}
By \ref{B}\ref{Bii}, each of these factors is contained in $B(1,K)$ and hence
\[
\frac{\ZZ{W',u_1,\ldots,u_d}{L',\ell,\ldots,\ell}(\widehat{H})}{\ZZ{W',u_1,\ldots,u_d}{L',j,\ldots,j}(\widehat{H})}
\in B(1,K)^d.
\] 
Finally note that since $\ell, j \in \mathcal{G}$ then $\ZZ{W',u}{L',\ell}(G) = \ZZ{W',u_1,\ldots,u_d}{L',\ell,\ldots,\ell}(\widehat{H})$ by \eqref{eq:G to H1} and \eqref{eq:star1} (since $\ZZ{}{}(S) = 1$ in this case), and similarly for $j$. 
So we deduce that the ratio $\ZZ{W',u}{L',\ell}(G)/\ZZ{W',u}{L',j}(G)$ is contained in $B(1,K)^d$, as desired.
This completes the proof.

\section{Finding constants for Theorems~\ref{thm:interval}}\label{sec:condition}
Recall that we denote the base of the natural logarithm by $e$.
\begin{proposition}
\label{pr:constants}
Let $\Delta\in \mathbb{N}_{\geq 3}$ and $k \geq e \Delta + 1$.
Then there exists $\varepsilon\in(0,1)$ such that with $K=1/\Delta$, for each $d=0,\ldots,\Delta-1$, with $b=\Delta-d$,
\begin{align}
\label{eq:repeat}
&0<\frac{(1+\varepsilon)^2}{(k-b)(1-K)^d-\varepsilon b}\leq K,\qquad
\end{align}
and additionally $\arcsin(K)\leq \frac{\pi}{3(\Delta-1+\eps)}$.
\end{proposition}
Note that \eqref{eq:repeat} is precisely the condition \eqref{eq:the conditions2} required in the hypothesis of Theorem~\ref{thm:ball2}. Therefore
combining this proposition with Theorem~\ref{thm:ball2} proves Theorem~\ref{thm:interval} for $c_{\Delta} = e$ (we only need part \ref{A}).

\begin{proof}
We first observe that once $\eps$ is set to zero in \eqref{eq:repeat} the condition states: for all $d=0,\ldots,\Delta-1$, 
\begin{equation}\label{eq:the conditions0}
\frac{1}{(k+d-\Delta)(1-K)^d}\leq K.
\end{equation}
We will show that \eqref{eq:the conditions0} is satisfied with strict inequality provided $k\geq e\Delta + 1$ when $K=1/\Delta$. Since the expression involving $\eps$ in~\eqref{eq:repeat} is a continuous, increasing function of $\eps$,  there exists an $\eps>0$ for which \eqref{eq:repeat} is satisfied. Moreover, for $K=1/\Delta$ we also have $\arcsin(K)<\frac{\pi}{3(\Delta-1)}$ for any $\Delta\geq 3$. So for $\eps>0$ small enough all conditions will be satisfied.

Noting that we assume $\Delta\ge3$, let us define for $d\geq 0$ the function
\[
f_\Delta(d):=\Delta\left(\frac{\Delta}{\Delta-1}\right)^d-d.
\]
We observe that condition \eqref{eq:the conditions0} with $K=1/\Delta$ is satisfied with strict inequality provided $f_\Delta(d)<k-\Delta$ for each $d=0,\ldots,\Delta-1$.

We first claim that $f_\Delta(d)$ as a function of $d$ is convex on $\mathbb{R}_{\geq 0}$.
Indeed, its second derivative in $d$ is given by
\[
  f_\Delta''(d)=\log^2\left(\frac{\Delta}{\Delta-1}\right)\left(f_d(\Delta)+d\right)>0.
\]
This implies that $f_\Delta$ attains its maximum on $[0,\Delta-1]$ at a boundary point: either when $d=0$ or when $d=\Delta-1$.
In fact the maximum is attained at $d=\Delta-1$, since we observe
\[
f_{\Delta}(\Delta-1) = \Delta\sum_{i=0}^{\Delta-1}\binom{\Delta-1}{i}\Big(\frac{1}{\Delta-1}\Big)^i-(\Delta-1) \geq 2\Delta -(\Delta-1) > \Delta = f_\Delta(0),
\]
where the penultimate inequality holds by taking the first two terms in the sum.
It now remains to check that $f_{\Delta}(\Delta-1)< k-\Delta$ whenever $k\geq e\Delta+1$. This holds since
\[
f_{\Delta}(\Delta-1)+\Delta-1=\Delta\left(1+\frac{1}{\Delta-1}\right)^{\Delta -1}<\Delta e.
\]
Thus we obtained that if $k\geq e\Delta+1$, we can choose $K=1/\Delta$, such that all the conditions in \eqref{eq:the conditions0} are satisfied with strict inequality.
This finishes the proof.
\end{proof}

\begin{rem}
We could have given a slightly tighter analysis by parametrising $K=x/\Delta$ in the proof given above.
However, it is not difficult to show that as $\Delta\to \infty$ the optimal choice for $x$ converges to $1$. In the next section we give better bounds for small values of $\Delta$ by adding additional constraints and using a computer to find the optimal value of $K$.
\end{rem}

\begin{rem}\label{rem:closed interval}
We note that if one replaces \eqref{eq:the conditions2} by 
\begin{align}\label{eq:the conditions alt}
0<\frac{(1+\varepsilon)^{2+b}}{(1-K)^d(k-b)-\varepsilon b(1+\eps)^b}\leq K
\end{align}
for all $d=0,\ldots,\Delta-1$ and $b=\Delta-d$, one can give essentially the same proof (where only \ref{C}\ref{Ciii} needs to be modified) to conclude that in Theorem~\ref{thm:interval} we can in fact guarantee an open set containing the closed interval $[0,1]$.
\end{rem}

\section{Improvements for small values of \texorpdfstring{$\Delta$}{Delta}}\label{sec:smallDelta}

In the previous section we showed that we can take $c_\Delta\leq e$ for each $\Delta\geq 3$ in the statement of Theorem~\ref{thm:interval}. In this section we prove the second part of Theorem~\ref{thm:interval}, by showing that for small values of $\Delta$, we can improve the bound on $c_\Delta$.
We do this by proving a slightly different version of Theorem~\ref{thm:ball2} in which we constrain the ratios of the restricted partition functions to lie in slightly different sets.

We first define a function $f$ by 
\[
f(d,K,\phi) = \max\Big( 2\sin \left(\phi/2\right), \sqrt{1+ (1+K)^{-2d} - 2 \cos(\phi)(1+K)^{-d}} \Big),
\]
where $d$ is a positive integer, $K\in (0,1)$, and $\phi$ is an angle.

\begin{theorem}
\label{thm:ballimprove}
Let $\Delta\in \mathbb{N}_{\geq 3}$.
Suppose that $k>\Delta$ and $0<\eps<1$ are such that there exist constants $K\in(0,1)$ with $\theta:=\arcsin(K) \in (0,\frac{\pi}{3(\Delta-1+\eps)})$ satisfying, for each $d=0,\ldots,\Delta-1$ with $b:=\Delta-d$, that
\begin{align}
\frac{(1+\varepsilon)^2(1+K)^d}{\cos((d+b\eps)\theta/2)(k-b)-\varepsilon b (1+K)^d} &\leq K  \qquad \text{ for } d = 0, \ldots, \Delta -2; \label{eq:new1}\\ 
\frac{(1+\varepsilon)(1+K)^d}{\cos((d+b\eps)\theta/2)(k-b)-\varepsilon b (1+K)^d} &\leq \frac{K}{f(d,K,(d+b\eps)\theta)} \hspace{1 cm} \nonumber\\
&\qquad \qquad \text{ for } d = 0, \ldots, \Delta -1. \label{eq:new2}
\end{align}
Then for each graph $G=(V,E)$ of maximum degree at most $\Delta$ and every $\var=(\var_e)_{e\in E}$ satisfying for each $e \in E$ that
\begin{itemize}
\item[(i)] $|\var_e|\leq \eps$, or 
\item[(ii)] $|\arg(\var_e)|\leq \eps \theta$ and $\eps <|\var_e|\leq 1$, 
\end{itemize} 
 the following statements hold for $\bZ(G) = \bZ(G;k, \var)$.
\begin{enumerate}[\bf{\Alph*'}]
\item\label{A'}
For all lists $W$ of distinct vertices of $G$ such that $W$ forms a leaf-independent set in $G$ 
and for all lists of pre-assigned colours $L$ of length $|W|$, $\ZZ{W}{L}(G)\neq 0$.
\item\label{B'}
For all lists $W=W'u$ of distinct vertices of $G$ such that $W$ is a leaf-independent set 
and for any two lists $L'\ell$ and $L'\ell'$ of length $|W|$:
\begin{enumerate}[\textbf{\textup{(\roman*)}}]
\item\label{B'i}
the angle between the vectors $\ZZ{W',u}{L',\ell}(G)$ and $\ZZ{W',u}{L',\ell'}(G)$ is at most $\theta$,
\item\label{B'ii}
\begin{equation}
\frac{|\ZZ{W',u}{L',\ell}(G)|}{|\ZZ{W',u}{L',\ell'}(G)|}\leq 1+K.
\end{equation}
\end{enumerate}
\item\label{C'}
For all lists $W=W'u$ of distinct vertices such that the initial segment $W'$ forms a leaf-independent set in $G$ 
and for all lists of pre-assigned colours $L'$ of length $|W'|$, the following holds. 
Write $\G=\G(G,W',L',u)$ and $\N=\N(G,W',L',u,\eps)$, let $d$ be the number of free neighbours of $u$, let $b=\Delta-d$, and let $m(\ell)$ be the number of fixed, bad neighbours of $u$ with pre-assigned colour $\ell$.
Then 
\begin{enumerate}[\textbf{\textup{(\roman*)}}]
\item\label{C'i}
for any $\ell\in \mathcal{G}\cup \N$, $\ZZ{W',u}{L',\ell}(G)\neq 0$, 
\item\label{C'ii}
for any $\ell,\ell'\in \mathcal{G}\cup \N$, the angle between the vectors $\ZZ{W',u}{L',\ell}(G)$ and $\ZZ{W',u}{L',\ell'}(G)$ is at most $(d+b\eps)\theta$,
\item\label{C'iii}
for any $\ell\in [k]$ and $j\in \G$
\begin{equation}
\frac{|\ZZ{W',u}{L',\ell}(G)|}{|\ZZ{W',u}{L',j}(G)|}\leq \eps^{m(\ell)}(1+K)^{d}.
\end{equation}
\end{enumerate}
\end{enumerate}
\end{theorem}

The proof is almost the same as the proof of Theorem~\ref{thm:ball2}. They key thing to do is to constrain the ratios \[\ZZ{W',u}{L',\ell}(G)/\ZZ{W',u}{L',\ell'}(G)\]
in such a way that allows for an induction.
To do this in this case, we essentially replace a convexity argument by an application of Lemma~\ref{lem:vectorsum}.
For convenience of the reader, we give a full proof.
\subsection{Proof}
We prove that \ref{A'}, \ref{B'}, and \ref{C'} hold by induction on the number of free vertices of a graph.
The base case consists of graphs with no free vertices. Clearly \ref{A'} and \ref{B'} hold in this case as they are both vacuous: if there are no free vertices then $W=V$, but then $W$ cannot be a leaf-independent set.

For statement \ref{C'} we note that since there are no free vertices, $V\setminus W' = \{u\}$, and hence $G$ must be a star with centre $u$. 
Part \ref{C'}\ref{C'i} follows since when $\ell\in \mathcal{G}\cup \N$ we have that  $\ZZ{W',u}{L',\ell}(G)$ is a product over nonzero edge-values.
Part \ref{C'ii} follows since, by changing the colour of $u$ from $\ell$ to $j\in \G\cup\N$, we obtain $\ZZ{W',u}{L',j}(G)$ from $\ZZ{W',u}{L',\ell}(G)$ by multiplying and dividing by at most $\deg(u)$ factors $w_{uv}$ with $\arg(w_{uv}) \leq\eps \theta$; hence the restricted partition function changes in angle by at most $\Delta \eps\theta \leq (d + b \eps)\theta$ (since $d=0$ and so $b = \Delta$).
Part \ref{C'iii} follows similarly, as when there are no free vertices we must have $d=0$, and changing the colour of $u$ from $j$ to $\ell$ corresponds to multiplying 
$\ZZ{W',u}{L',j}(G)$ by at most $\deg(u)$ factors $w_{uv}$ all satisfying $|w_{uv}|\leq 1$ and $m(\ell)$ of the factors satisfying $|w_{uv}| \leq \eps$.

Now let us assume that statements \ref{A'}, \ref{B'}, and \ref{C'} hold for all graphs with $f\geq 0$ free vertices. We wish to prove the statements for graphs with $f+1$ free vertices. 
We start by proving \ref{A'}.

\subsubsection{Proof of \ref{A'}}
Let $u$ be any free vertex.
We proceed using the fact that $|\ZZ{W}{L}(G)|=|\sum_{j=1}^k\ZZ{W,u}{L,j}(G)|$. 
Let $\G=\G(G,W,L,u)$, $\B=\B(G,W,L,u,\eps)$, $\N=[k]\setminus(\G\cup\B)$ and $\hat b=|\B|$.
Let $d$ be the number of free neighbours of $u$ and let $b=\Delta-d$. Note that $\hat b\leq b$ and $|\G|\geq k-b$. 
After fixing $u$ to any $j\in[k]$ we have one less free vertex, and hence can apply \ref{C'} using induction as necessary. 

There are two cases to consider. If $\hat b=0$ then 
by induction using \ref{C'}\ref{C'i} we have that the $\ZZ{W,u}{L,j}(G)$ are non-zero and by \ref{C'}\ref{C'ii}  the angle between any two of the $\ZZ{W,u}{L,j}(G)$ is at most $\Delta\theta<\pi/2$.
Hence by Lemma~\ref{lem:vectorsum}
\[
|\ZZ{W}{L}(G)|=\Big|\sum_{j=1}^k \ZZ{W,u}{L,j}(G)\Big|\geq \cos(\pi / 4)\sum_{j=1}^k \big|\ZZ{W,u}{L,j}(G)\big| > 0\,.
\]
If $\hat b>0$ then $u$ must have at least one fixed neighbour, and hence $d\le\Delta-1$. Let 
\begin{align*}
M:=&\min\big\{\big|\ZZ{W,u}{L,j}(G)\big|: j \in \G \big\}, \\
m:=& \max\big\{\big|\ZZ{W,u}{L,j}(G)\big|: j\in\mathcal{B} \big\},
\end{align*}
and assume that $j_M \in \G$ achieves the minimum above and $j_m \in \B$ achieves the maximum. Note that $M>0$ by induction using \ref{C'}\ref{C'i}.  
Note further by induction using \ref{C'}\ref{C'iii} that 
\begin{equation}
\label{eq:m}
m  \leq 
\varepsilon^{m(j_m)}(1+K)^{d} |\ZZ{W',u}{L',j_M}(G)| \leq \varepsilon(1+K)^{d} M,
\end{equation}
where we used that $m(j_m)\geq 1$ since $j_m \in \mathcal{B}$.  We then have
 
\begin{align*}
|\ZZ{W}{L}(G)|=\Big|\sum_{j=1}^k \ZZ{W,u}{L,j}(G)\Big|
  &\ge \Big|\sum_{j\in\G\cup\N} \ZZ{W,u}{L,j}(G)\Big| - \sum_{j\in\B} |\ZZ{W,u}{L,j}(G)|
  \\
  &\ge \cos((d+b\eps)\theta/2) \sum_{j\in \G\cup \N}| \ZZ{W,u}{L,j}(G)|-  \sum_{j\in\B} |\ZZ{W,u}{L,j}(G)|
\\&\ge M|\G|\cos((d+b\eps)\theta/2) - m|\B|
\\&\geq M\big[(k-b)\cos((d+b\eps) \theta/2)-b \varepsilon(1+K)^d\big]\,,
\end{align*}
where the first inequality is the triangle inequality, the second uses \ref{C'}\ref{C'ii} and Lemma~\ref{lem:vectorsum}, the third uses the definition of $M$ and $m$, and the fourth follows from \eqref{eq:m}. 
Now the conditions~\eqref{eq:new1} give that $\ZZ{W}{L}(G)\ne0$ (recalling we have $d \leq \Delta - 1$ and noting the denominator in \eqref{eq:new1} must be positive).

Next we will prove \ref{B'}.

\subsubsection{Proof of \ref{B'}}
The proof starts in exactly the same way as the proof of \ref{B}.
Recall that $\deg(u)=1$ and that its unique neighbour, which we call $v$, is free.
We start by introducing some notation.

We define complex numbers $z_j$ for $j\in [k]$ by
\begin{equation}
z_j:=\ZZ{W',u,v}{L',\ell,j}(G\setminus uv)=\ZZ{W',u,v}{L',\ell',j}(G\setminus uv)
= \ZZ{W',v}{L',j}(G - u) \,,
\end{equation}
where the equalities follow because $u$ is isolated in $G\setminus uv$ and so makes no contribution to the partition function. 
Let $w:=w_{uv}$ and define complex numbers $x_j$ and $y_j$ for $j\in [k]$ by
\begin{align*}
x_j&=\begin{cases}
z_j & \text{ if } j\neq \ell;\\
\var z_\ell &\text{ if } j=\ell\,,
\end{cases}&
y_j=\begin{cases}
z_j & \text{ if } j\neq \ell';\\
\var z_{\ell'} &\text{ if } j=\ell'\,.
\end{cases}
\end{align*}
Let $x=\sum_{j=1}^k x_j$ and $y=\sum_{j=1}^k y_j$.
Observe that $x=\ZZ{W',u}{L',\ell}(G)$ and $y=\ZZ{W',u}{L',\ell'}(G)$, and that we may apply induction to the restricted partition function evaluations represented by the $z_j$ because there are $f$ free vertices in $G\setminus uv$ when the vertices in $W'uv$ are fixed. 

For \ref{B'}\ref{B'i} and \ref{B'ii} we wish to bound the angle between $x$ and $y$ and the ratio $|x|/|y|$ respectively.
To do this we first bound $|y|$ and $|x-y|$. 

We note for later that by the definition of $x$ and $y$, we have $x-y=(\var-1)z_\ell+(1-\var)z_{\ell'}=(1-\var)(z_{\ell'}-z_\ell)$. Also $|1 - \var| \leq 1+ \eps$ by conditions (i) and (ii) in the statement of the theorem so $|x-y| \leq (1+\eps)|z_{\ell'}-z_\ell|$.

Let $\G=\G(G,W'u,L'\ell',v)$, $\B=\B(G,W'u,L'\ell',v,\eps)$, $\N=[k]\setminus(\G\cup\B)$, $\hat b=|\B|$, and suppose $v$ has $d$ free neighbours (in $G$ when $W'u$ is fixed), so $d = \Delta - 1$. 
Let $H$ be the graph obtained from $G$ by deleting all fixed neighbours of $v$, i.e.\ $H=G-(N_G(v)\cap W)$, and let $W''=W\setminus N_G(v)$ and $L''$ be the sublist of $L$ corresponding to the vertices in $W''$. 
Observe that by definition for any $j\in [k]$, we have the similar identities
\begin{equation}\label{eq:from G to H improved}
z_j:= \ZZ{W',v}{L', j}(G - u)=\ZZ{W'',v}{L'',j}(H)\cdot \prod_{\substack{v'\in W'\cap N_G(v)\\ \text{s.t. } L'(v')=j}}w_{vv'},
\end{equation}
\begin{equation}\label{eq:from G to H improved2}
\ZZ{W',u,v}{L',\ell',j}(G)=\ZZ{W'',v}{L'',j}(H)\cdot \prod_{\substack{v'\in W \cap N_G(v)\\ \text{s.t. } L'(v')=j}}w_{vv'},
\end{equation}
where by $L'(v')$ we mean the colour that the list $L'$ assigns to the vertex $v'$.
In particular, if $j\in \G$, we have $\ZZ{W',v}{L',j}(G-u)=\ZZ{W'',v}{L'',j}(H)$. 

Now write $b=\Delta-d$. Note that $d\le\Delta-1$, and define $M$, $j^*$, and $C$ by
\begin{align*}
M&:=\min\big\{\big|\ZZ{W'',v}{L'',j}(H)\big|: j\in [k] \big\}=\ZZ{W'',v}{L'',j^*}(H);\\
C&:=(k-b)\cos((d+b\eps)\theta/2)- b\varepsilon (1+K)^d\,.
\end{align*}
Note that for all $j \in [k]$, we have 
\begin{equation}
\label{eq:HM}
M \leq |\ZZ{W'',v}{L'',j}(H)| \leq M(1+K)^d,
\end{equation}
 where the upper bound follows by induction using \ref{C'}\ref{C'iii}
(here $H$ with fixed vertices $W''v$ has fewer free vertices than $G$ with fixed vertices $W$) and noting that all colours in $[k]$ are good for $v$ in $H$. We perform a similar calculation to the case $\hat b>0$ of \ref{A'} to bound $|y|$. 
We have
\begin{align*}
|y| = \Big|\ZZ{W',u}{L',\ell'}(G)\Big| 
  &\ge \Big|\sum_{j\in\G\cup\N}\ZZ{W',u,v}{L',\ell',j}(G)\Big|-\sum_{j\in\B}|\ZZ{W',u,v}{L',\ell',j}(G)|
\\  &\ge \cos((d+b\eps)\theta/2) \sum_{j\in\G}|\ZZ{W',u,v}{L',\ell',j}(G)| -\sum_{j\in\B}|\ZZ{W',u,v}{L',\ell',j}(G)|
\\  &\ge \cos((d+b\eps)\theta/2) \sum_{j\in\G}|\ZZ{W',v}{L',j}(H)| - \eps\sum_{j\in\B}|\ZZ{W',v}{L',j}(H)|  
\\&\ge M|\G|\cos((d+b\eps)\theta/2) - M|\B|\eps(1+K)^d \ge M\cdot C\,;
\end{align*}
here the first inequality is the triangle inequality, the second follows from Lemma~\ref{lem:vectorsum} and induction using \ref{C'}\ref{C'ii}, the third follows from \eqref{eq:from G to H improved2} and the fourth follows from \eqref{eq:HM}. 

We next claim that
\begin{equation}\label{eq:bound x-y over y}
\frac{|x-y|}{|y|}\leq K.
\end{equation}
To prove this, we will need to distinguish three cases, for which we now introduce the notation. 
Let $\widehat\G=\G(G-u,W',L',v)$, $\widehat\B=\B(G-u,W',L',v,\eps)$, $\widehat\N= \N(G-u,W',L',v,\eps)$, and let $\widehat m(j)$ be the number of bad neighbours of $v$ in $G-u$ with pre-assigned colour $j$. 
Then, by \eqref{eq:from G to H improved} and \eqref{eq:HM} we have for any $j$, 
\begin{align}
\label{eq:Zlnew1}
|z_j| &\le \varepsilon^{\widehat m(j)}|\ZZ{W'',v}{L'',j}(H)|\le \varepsilon^{\widehat m(j) }M(1+K)^d.
\end{align}

We now come to the three cases: either (a) $\ell ,\ell' \in \widehat\G$, or (b) $\ell \in \widehat\B$ or $\ell' \in \widehat\B$, or (c) both $\ell$ and $\ell'$ are contained in $\widehat\G\cup \widehat\N$ and one of them is not contained in $\widehat\G$.

In case (a), by induction using \ref{C'} for $\ZZ{W',v}{L',j}(G-u) = \ZZ{W',v}{L',j}(G \setminus uv)=z_j $, the angle between $z_{\ell}$ and $z_{\ell'}$ is at most $(d+b\eps)\theta\le (\Delta - 1 + \eps) \theta \le\pi/3$ (by \ref{C'}\ref{C'ii}) and $(1+K)^{-d} \leq |z_{\ell} / z_{\ell'}| \leq (1+K)^d$ by \ref{C'}\ref{C'iii}, so we can apply Lemma~\ref{lem:geom} to conclude that 
\[
|z_{\ell'}-z_\ell|\leq f(d,K,(d+b\eps)\theta)\cdot \max\{|z_\ell|,|z_{\ell'}|\} \leq (1+K)^dM,
\]
where the final inequality follows by \eqref{eq:Zlnew1}. Then
\begin{align*}
\frac{|x-y|}{|y|} 
\le \frac{(1 + \eps) |z_{\ell'} - z_{\ell}| }{|y|}
\le (1+\varepsilon) \frac{f(d,K,(d+b\eps)\theta)M(1+K)^d}{M\cdot C}< K\,,
\end{align*}
where the final inequality comes from the condition \eqref{eq:new2}.
 
For case (b), at least one of $\ell,\ell'$ is in $\widehat\B$, so we know that $d\leq \Delta-2$ (as the degree of $v$ in $H$ is at most $\Delta-1$ and $v$ has at least one fixed neighbour in $H$). 
We use the triangle inequality to obtain
\begin{align*}
|z_\ell-z_{\ell'}|\le |z_\ell| + |z_{\ell'}|
\le \big(\varepsilon^{\widehat m(\ell)} +\varepsilon^{\widehat m(\ell')} \big)M(1+K)^d
\le (1+\varepsilon)M(1+K)^d,
\end{align*}
using \eqref{eq:Zlnew1} for $j\in\{\ell,\ell'\}$ and that at least one of $\widehat m(\ell)$ and $\widehat m(\ell')$ is at least $1$ in this case. Therefore
\begin{equation*}
 \frac{|x-y|}{|y|}
 = \frac{(1 + \eps) |z_{\ell'} - z_{\ell}| }{|y|}
  \leq   \frac{(1+\varepsilon)^2M(1+K)^d}{MC}< K,
\end{equation*}
where the final inequality comes from the condition \eqref{eq:new1} (noting that we only need the condition for $d = 1, \ldots, \Delta - 2$).

In case (c), both $\ell$ and $\ell'$ are contained in $\widehat\G\cup \widehat\N$ and at least one of them is contained in $\widehat{N}$, and so $d\leq \Delta-2$ (as in case (b)). By induction using \ref{C'} for $\ZZ{W',v}{L',j}(G-u) = \ZZ{W',v}{L',j}(G \setminus uv)=z_j $, the angle between $z_{\ell}$ and $z_{\ell'}$ is at most $(d+b\eps)\theta \le (\Delta - 1 + \eps)\theta \le \pi/3$ by \ref{C'}\ref{C'ii}. Thus 
using Proposition~\ref{prop:simplegeom} and \eqref{eq:Zlnew1}, we obtain
\begin{align*}
|z_\ell-z_{\ell'}| \le \max\{|z_\ell|, |z_{\ell'}|\} \le M(1+K)^d,
\end{align*}
and so as before
\begin{equation*}
 \frac{|x-y|}{|y|}
 = \frac{(1 + \eps) |z_{\ell'} - z_{\ell}| }{|y|}
  \leq   \frac{(1 + \eps)(1+K)^dM}{MC}< K,
\end{equation*}
where the final inequality comes from the condition \eqref{eq:new1} (noting that we only need the condition for $d = 1, \ldots, \Delta - 2$).
This establishes \eqref{eq:bound x-y over y} in all cases.

Now, by Proposition~\ref{prop:simplegeom}, the angle $\gamma$ between $x$ and $y$ satisfies $\sin \gamma\le |x-y|/|y|\leq K$, and we conclude that $\gamma\le\arcsin(K)=\theta$ as required for \ref{B'}\ref{B'i}. 
Additionally, by the triangle inequality we have 
\[
\frac{|x|}{|y|}\leq \frac{|y|+|x-y|}{|y|}\leq 1+K\,,
\]
which gives \ref{B'}\ref{B'ii}. We now turn to \ref{C'}.

\subsubsection{Proof of \ref{C'}}
We start with \ref{C'i}, that is we will show that for any $\ell\in\G \cup \N$, $\ZZ{W',u}{L',\ell}(G)\neq 0$. Since we have already proved \ref{A'} and \ref{B'} for the case of $f+1$ free vertices and since we have $f+1$ free vertices for $\ZZ{W',u}{L',\ell}(G)$, we might hope to immediately apply \ref{A'}; the only problem is that $W'u$ may not a leaf- independent set, so we will modify $G$ first. 

Let $v_1,\dots,v_d$ be the free neighbours of $u$. We construct a new graph $H$ from $G$ by adding vertices $u_1, \ldots, u_d$ to $G$ and replacing each edge $uv_i$ with $u_iv_i$ for $i=1, \ldots, d$, while keeping all other edges of $G$ unchanged (so note that $u$ is only adjacent to its fixed neighbours in $H$). Each edge $e$ of $H$ is assigned value $\var'_e$ where $\var'_e = \var_e$ if $e$ is an edge of $G$ and $\var'_{u_iv_i} = \var_{uv_i}$ for the new edges $uv_i$. 
See Figure~\ref{fig:ballH} for an illustrative example.

Then by construction we have 
\begin{equation}\label{eq:G to H}
\ZZ{W',u}{L',\ell}(G)=\ZZ{W',u,u_1,\ldots,u_d}{L',\ell,\ell,\ldots,\ell}(H).
\end{equation}

Notice that in $H$, the vertex $u$ together with its neighbours form a star $S$ that is disconnected from the rest of $H$ (and all vertices of $S$ are in $W = W'u$ so they are fixed). Thus $H$ is the disjoint union of $S$ and some graph $\widehat{H}$. 
Thus the partition function $z:=\ZZ{W',u,u_1,\ldots,u_d}{L',\ell,\ell,\ldots,\ell}(H)$ factors as 
\begin{equation}
\label{eq:starimproved}
\ZZ{W',u,u_1,\ldots,u_d}{L',\ell,\ell,\ldots,\ell}(H)
= \ZZ{W',u_1,\ldots,u_d}{L',\ell,\ldots,\ell}(\widehat{H}) \cdot
\ZZ{W',u}{L',\ell}(S);
\end{equation}
here we abuse notation by having a list $W'u_1\ldots u_d$ (resp. $W'u$) that may contain vertices not in $\widehat{H}$ (resp. $S$); such vertices and their corresponding colour should simply be ignored.

The fixed vertices in $\widehat{H}$ form a leaf-independent set, so we can apply $\ref{A'}$ to conclude that the first factor above is nonzero. It is also clear that second factor above is nonzero because all vertices in $S$ are fixed and $\ell\in \mathcal{G}\cup \N$. Hence $z \not=0$ as required.

To prove part \ref{C'ii}, we will apply \ref{B'} to $\widehat{H}$ with $W'uu_1\ldots u_d$ fixed, which (as above) is possible since we already proved \ref{B'} for $f+1$ free vertices and $W'uu_1\ldots u_d$ restricted to $\widehat{H}$ is a leaf-independent set. 
By \ref{B'}\ref{B'i} the angle between
\[
\ZZ{W',u_1,\ldots,u_{d-1},u_d}{L',\ell,\ldots,\ell,\ell'}(\widehat{H})\qquad\text{and}\qquad\ZZ{W',u_1,\ldots,u_{d-1},u_d}{L',\ell,\ldots,\ell,\ell}(\widehat{H})
\]
is at most $\theta$. Continuing to change the label of each $u_i$ one step at the time, we conclude that the angle between
\[
\ZZ{W',u_1,\ldots,u_d}{L',\ell',\ldots,\ell'}(\widehat{H})\qquad\text{and}\qquad\ZZ{W',u_1,\ldots,u_d}{L',\ell,\ldots,\ell}(\widehat{H})
\]
is at most $d\theta$.
We next notice that since for \ref{C'ii} we assume $\ell,\ell' \in \G \cup \N$, changing the colour of $u$ from $\ell$ to $\ell'$ can only change $\ZZ{W',u}{L',\ell}(S)$ by $\deg(u) \leq \Delta-d=b$ factors, each of argument at most $\eps \theta$ thus giving a total change of angle by at most $b \eps \theta$.
Hence by \eqref{eq:starimproved}, we therefore conclude that the angle between $\ZZ{W',u}{L',\ell}(G)$ and $\ZZ{W',u}{L',\ell'}(G)$ is at most $d\theta+b\eps\theta$.

To prove \ref{C'iii} we observe that we can write for any $j,\ell\in [k]$ the telescoping product,
\begin{equation}\label{eq:telescope prime improved}
\frac{\ZZ{W',u_1,\ldots,u_d}{L',\ell,\ldots,\ell}(\widehat{H})}{\ZZ{W',u_1,\ldots,u_d}{L',j,\ldots,j}(\widehat{H})}=
\frac{\ZZ{W',u_1,\ldots,u_d}{L',\ell,\ldots,\ell}(\widehat{H})}{\ZZ{W',u_1,\ldots, u_{d-1},u_d}{L',\ell,\ldots,\ell,j}(\widehat{H})}\cdots
\frac{\ZZ{W',u_1,u_2,\ldots,u_d}{L',\ell,j,\ldots,j}(\widehat{H})}{\ZZ{W',u_1,\ldots,u_d}{L',j,\ldots,j}(\widehat{H})}\,,
\end{equation}
and consequently by \ref{B'}\ref{B'ii}, it follows that 
\begin{equation}\label{eq:consequence of telescoping prime improved}
(1+ K)^{-d}\le \frac{|\ZZ{W',u_1,\ldots,u_d}{L',\ell,\ldots,\ell}(\widehat{H})|}{|\ZZ{W',u_1,\ldots,u_d}{L',j,\ldots,j}(\widehat{H})|}\le (1+ K)^d.
\end{equation}
Next we observe that in $S$ when changing the colour of $u$  from $\ell\in[k]$ to a good colour $j\in\G$, we have
\[
\big|\ZZ{W',u}{L',\ell}(S)\big|\le\varepsilon^{m(\ell)}\big|\ZZ{W',u}{L',j}(S)\big|\,,
\]
and so by \eqref{eq:starimproved} we have
\[
\big|\ZZ{W',u,u_1,\ldots,u_d}{L',\ell,j,\ldots,j}(H)\big|\le\varepsilon^{m(\ell)}\big|\ZZ{W',u,u_1,\ldots,u_d}{L',j,j,\ldots,j}(H)\big|\,,
\]
Combining the above inequality with \eqref{eq:consequence of telescoping prime improved}, we obtain \ref{C'iii} as required:
\[
\frac{|\ZZ{W',u}{L',\ell}(G)|}{|\ZZ{W',u}{L',j}(G)|}=
\frac{|\ZZ{W',u,u_1,\ldots,u_d}{L',\ell,\ell,\ldots,\ell}(H)|}{|\ZZ{W',u,u_1,\ldots,u_d}{L',j,j,\ldots,j}(H)|}\le
\eps^{m(\ell)}(1+K)^{d}.
\]
This completes the proof.

\quad 
\\
In Table~\ref{table:small_delta_improved} we list the improvements that Theorem~\ref{thm:ballimprove} gives; we give the improved values for $k_\Delta$ together with values of $K$ and $\theta$ that allow the reader to check that the conditions of Theorem~\ref{thm:ballimprove} are met with strict inequality for $\eps=0$.


\begin{table}[h]
	\caption{Bounds for the number of colours and values for $K$ and $\theta$ for small values of $\Delta$ for Theorem~\ref{thm:ballimprove}.\label{table:small_delta_improved}}
	\begin{center}
	\begin{tabular}{cllc}
		\toprule
	$\Delta$ & $K$ & $\theta$ & $k_\Delta$ \\
	\midrule
	3  &  0.4124  &  0.4251  &  6 \\
	4  &  0.2900  &  0.2943  &  8 \\
	5  &  0.2224  &  0.2244  &  11 \\
	6  &  0.1814  &  0.1826  &  14 \\
	7  &  0.1536  &  0.1543  &  17 \\
	8  &  0.1334  &  0.1339  &  20 \\
	9  &  0.1179  &  0.1183  &  23 \\
	10  &  0.1057  &  0.1060  &  26 \\
	11  &  0.0959  &  0.0961  &  29 \\
	12  &  0.0877  &  0.0879  &  32 \\
	13  &  0.0800  &  0.0802  &  35 \\
	\bottomrule
	\end{tabular}
	\end{center}
	\end{table}

\section{Concluding remarks and questions}\label{sec:conclusion}
In the present paper we have established that if $k$ is an integer satisfying $k\geq e\Delta+1$, then there exists an open set $U$ containing $[0,1)$ such that the for any graph $G$ of maximum degree at most $\Delta$ and $w\in U$, $\bZ(G;k,w)\neq 0$.
For small values of $\Delta$ we have shown that we can significantly improve on $e$.
We raise the following question.
\begin{question}
Is it true that for each $\Delta\in\mathbb{N}_{\geq 3}$ there exists an open set $U=U_\Delta$ containing $[0,1]$ such that for any integer $k$ satisfying $k\geq \Delta+1$ and any graph $G$ of maximum degree at most $\Delta$ and $w\in U$, $\bZ(G;k,w)\neq 0$?
\end{question}

As mentioned in the introduction, Barvinok's approach for proving zero-free regions for partition functions has been used for several types of partition functions, see~\cite{Ba16,BaS17,B17,R18}.
This of course raises the question as to which of these partition functions our ideas could be applied.
In particular it would be interesting to apply our ideas to partition functions of edge-colouring models (a.k.a.\ tensor networks, or Holant problems). This framework may be useful to study the zeros of the Potts model on line graphs.

Implicit in our proof of Theorem~\ref{thm:ball2} is an iteration of a complex-valued dynamical system, which for $k=2$ essentially coincides with the dynamical system analysed in \cite{LSS18,PR18b}.
Given the recent success of the use of methods from the field of complex dynamical systems to identify zero-fee regions and the location of zeros of the partition function of the hardcore model~\cite{PR18a} and the partition function of the Ising model~\cite{LSS18,PR18b}, it seems natural to study this dynamical system.
 We intend to expand on this in future work.

\section*{Acknowledgements}
We are grateful to two anonymous referees for helpful comments on the relation with statistical physics.

\end{document}